\documentclass[12pt, reqno]{amsart}
\usepackage{amsmath, amsthm, amscd, amsfonts, amssymb, graphicx, color, mathrsfs}
\usepackage[bookmarksnumbered, colorlinks, plainpages]{hyperref}

\usepackage{slashed}
\usepackage[capitalize]{cleveref}

\newtheorem{theorem}{Theorem}[section]

\theoremstyle{definition}
\newtheorem{definition}[theorem]{Definition}

\theoremstyle{remark}
\newtheorem{remark}[theorem]{Remark}
\numberwithin{equation}{section}

\newcommand*\diff{\mathop{}\!\mathrm{d}}
\DeclareMathOperator{\supp}{supp}
\DeclareMathOperator*{\esssup}{ess\,sup}

\crefname{equation}{}{}

\begin{document}
\setcounter{page}{1}

\title[Estimates for operators on the torus. I]{ Estimates for pseudo-differential operators on the torus revisited. I}

\author[D. Cardona]{Duv\'an Cardona}
\address{
  Duv\'an Cardona:
  \endgraf
  Department of Mathematics: Analysis, Logic and Discrete Mathematics
  \endgraf
  Ghent University, Belgium
  \endgraf
  {\it E-mail address} {\rm duvanc306@gmail.com, duvan.cardonasanchez@ugent.be}
  }

\author[M. A. Mart\'inez]{Manuel Alejandro Mart\'inez}
\address{
  Manuel Alejandro Mart\'inez
  \endgraf
  Department of Mathematics
  \endgraf
 Universidad del Valle de Guatemala, Guatemala
  \endgraf
  {\it E-mail address} {\rm mar21403@uvg.edu.gt, manuelalejandromartinezf@gmail.com}
  }

\thanks{ Manuel Alejandro Mart\'inez has been supported by the {\it{Liderazgo en Ciencias}} scholarship of the Universidad del Valle de Guatemala.  Duv\'an Cardona is supported  by the FWO  Odysseus  1  grant  G.0H94.18N:  Analysis  and  Partial Differential Equations, by the Methusalem programme of the Ghent University Special Research Fund (BOF)
(Grant number 01M01021) and has been supported by the FWO Fellowship
Grant No 1204824N of the Belgian Research Foundation FWO}


\begin{abstract}
In this paper we prove $L^p$-estimates for H\"ormander classes of pseudo-differential operators on the torus $\mathbb{T}^n$. The results are presented in the context of the global symbolic calculus of Ruzhansky and Turunen on $\mathbb{T}^n\times\mathbb{Z}^n$ by using the discrete Fourier analysis on the torus which extends the usual ($\rho,\delta$)-Hörmander classes on $\mathbb{T}^n$. The main results extend \'Alvarez and Hounie's method for $\mathbb{R}^n$ to the torus, and Fefferman's $L^p$-boundedness theorem in the toroidal setting allowing the condition $\delta\geq\rho$. When $\delta \leq \rho$, our results recover the available estimates in the literature.
\end{abstract} 

\maketitle
\tableofcontents
\allowdisplaybreaks

\section{Introduction} In this paper we address the problem of the boundedness of pseudo-differential operators on the torus, $\mathbb{T}^n := \mathbb{R}^n/\mathbb{Z}^n$, associated to symbols in the periodic H\"ormander classes $S^m_{\rho,\delta}(\mathbb{T}^n\times \mathbb{Z}^n )$. Since the torus is a compact manifold without boundary, these classes agree with the standard ones defined by H\"ormander on the torus $\mathbb{T}^n,$ when $0\leq \delta<\rho\leq 1$ and $\rho\geq 
1-\delta,$ see Hörmander \cite{hormander}. On the other hand, the periodic H\"ormander classes differ from the ones defined by local coordinate systems, as in H\"ormander approach when $\delta\geq \rho$ , or when $\rho<1-\delta,$ (here we refer to the periodic H\"ormander classes to the ones introduced by Ruzhansky and Turunen \cite[Chapter~4]{ruzhansky-turunen}).  In consequence, the problem addressed in this paper is classical with several results available in the literature transferred from the classical boundedness properties of pseudo-differential operators in the Euclidean space, see e.g. Fefferman \cite{fefferman-Lp}, \'Alvarez and Hounie \cite{alvarez-hounie}, to any compact manifold without boundary with $0\leq \delta<\rho\leq 1$ and $\rho\geq 1-\delta.$ Note that this condition implies that $\rho>1/2,$ leaving a range of possible improvements for boundedness criteria when $\rho\leq 1/2,$ or more generally when $\rho<1-\delta.$

We recall that in the Euclidean case, Calder\'on and Vaillancourt proved that pseudo-differential operators with symbols in the class $S^{0}_{\rho,\rho}(\mathbb{R}^n\times\mathbb{R}^n)$ are bounded on $L^2(\mathbb{R}^n)$ for some $0\leq\rho<1$, see \cite{calderon-vaillancourt-1, calderon-vaillancourt-2}. This result cannot be extended when $\rho=1$, namely, there exist symbols in $S^{0}_{1,1}(\mathbb{R}^n\times\mathbb{R}^n)$ whose corresponding pseudo-differential operators are not bounded on $L^2$, see e.g. \cite{duoandikoetxea} for a classical argument of this fact due to H\"ormander. Moreover, Fefferman \cite{fefferman-Lp} proved the $L^\infty(\mathbb{R}^n)$-$\mathrm{BMO}(\mathbb{R}^n)$ boundedness of pseudo-differential operators with symbols in the class $S^{m}_{\rho,\delta}(\mathbb{R}^n\times\mathbb{R}^n)$, for $m=-n(1-\rho)/2$ where $0\leq \delta<\rho\leq 1$. Fefferman also obtained the $L^p(\mathbb{R}^n)$ boundedness of these classes for $m\leq -n(1-\rho)|1/p - 1/2|$ and $1<p<\infty$. In view of classical examples due to Wainger and Hirschman, Fefferman's result is sharp for Fourier multipliers. We also remark that the historical development of the problem of the $L^p$-boundedness of pseudo-differential operators has been discussed on  $\mathbb{R}^n$ e.g. in \cite{nagase, wang}.

Pseudo-differential operators with symbols in Hörmander classes can be defined in $C^\infty$-manifolds by using local charts. Hence, we shall view the torus $\mathbb{T}^n:=\mathbb{R}^n/\mathbb{Z}^n$ as a quotient additive group and a $n$-manifold, with the prefered atlas of local coordinate systems given by the restriction mapping $x\mapsto x + \mathbb{Z}^n$ to the open sets $\Omega \subset \mathbb{R}^n$, see McLane \cite{mclane}. We notice that in \cite{agranovich}, Agranovich gives a global definition of pseudo-differential operators on the circle $\mathbb{S}^1=\mathbb{T}^1$, instead of the local formulation considering the circle as a manifold.  By using the Fourier transform, this definition was extended to the torus $\mathbb{T}^n$. Moreover, it has been proved that the $(\rho,\delta)$-classes by Agranovich and Hörmander are equivalent, namely, this is the equivalence theorem due to  McLane \cite{mclane}. Here, we consider toroidal pseudo-differential operators in the context of the pseudo-differential calculus on the torus developed by Ruzhansky, Turunen and Vainniko \cite{ruzhansky-turunen2, ruzhansky-turunen}. Also, $L^p$-bounds on the circle that can be extended to the torus can be found in \cite{wong} in the classical setting of the Calder\'on-Zygmund theory. On the other hand, the toroidal analogous of Fefferman's result was proved by Delgado in \cite{delgado} for the torus and extended to compact Lie groups by Delgado and Ruzhansky in \cite{delgado-ruzhansky}, however is still required that $\delta<\rho$. In \cite{cardona-ruzhansky-subelliptic} this result has been extended to the borderline $0\leq \delta\leq \rho$ for subelliptic Hörmander classes on compact Lie groups. For other works on the $L^p$-boundedness of pseudo-differential operators we refer the reader to \cite{cardona, molahajloo-wong, ruzhansky-turunen-quant}.

In this paper we show that, as in the Euclidean case, the boundedness of pseudo-differential operators for the classes $S^m_{\rho,\delta}(\mathbb{T}^n\times \mathbb{Z}^n )$  can be analysed even if one considers Hounie's approach to this problem, namely, if the parameters $\rho$ and $\delta$ are considered satisfying the conditions $0\leq \delta<1,$ and $0<\rho\leq 1,$ allowing for the forbidded range $0\leq \rho\leq 1/2.$ 

Note that, for $0\leq \delta<1,$ and $0<\rho\leq 1,$ \'Alvarez and Hounie \cite{alvarez-hounie}, proved the $L^p(\mathbb{R}^n)$-$L^q(\mathbb{R}^n)$ boundedness of  pseudo-differential operators for $p\leq q$ even when $\delta \geq \rho$. Here we state their result where $\Psi^m_{\rho, \delta}(\mathbb{R}^n \times \mathbb{R}^n)$ denotes the Hörmander class of pseudo-differential operators on $\mathbb{R}^n$ \cite{hormander}.

\begin{theorem}
    Let $T \in \Psi^m_{\rho, \delta}(\mathbb{R}^n \times \mathbb{R}^n) $, let $0 < \rho \leq 1$, let $0 \leq \delta < 1$ and let $\lambda = \max\{ 0, (\delta-\rho)/2 \}$. Then:
    \begin{itemize}
        \item[(a)] if $m \leq - n [(1-\rho)/2 + \lambda]$, the operator $T$ extends to a bounded operator from the Hardy space $H^1(\mathbb{R}^n)$ into $L^1(\mathbb{R}^n)$, from $L^1(\mathbb{R}^n)$ into weak-$L^1(\mathbb{R}^n)$,\footnote{Here, as usual, for $Z=\mathbb{R}^n$ or $Z=\mathbb{T}^n,$ the weak-$L^1(Z)$ space is determined by the seminorm $\|f\|_{weak-L^1(Z)} := \sup_{\lambda > 0} \lambda|\{x \in Z : |f(x)| > \lambda \}| < \infty.$} and from $L^\infty(\mathbb{R}^n)$ into $\mathrm{BMO}(\mathbb{R}^n)$.\\
        \item[(b)] if $m \leq -n[(1-\rho)|1/p - 1/2| + \lambda]$, then $T$ extends to a bounded operator on $L^p(\mathbb{R}^n)$ for $1<p<\infty$.
    \end{itemize}    
    \label{theo:mainR}
\end{theorem}

In this paper, by following the approach of \'Alvarez and Hounie \cite{alvarez-hounie}, we begin by proving some useful estimates for the Schwartz kernel of toroidal pseudo-differential operators. Then, we use these results to prove the $L^\infty(\mathbb{T}^n)$-$\mathrm{BMO}(\mathbb{T}^n)$ and the $H^1(\mathbb{T}^n)$-$L^1(\mathbb{T}^n)$ boundedness properties of these operators. These estimates solve in the torus some open problems addressed in \cite{cardona-ruzhansky-subelliptic}. Finally, we employ the Fefferman-Stein interpolation argument to obtain the $L^p(\mathbb{T}^n)$-boundedness of our class for $1<p<\infty$. Moreover, we prove the toroidal analogous of a result  by \'Alvarez and Milman \cite[Theorem~4.1]{alvarez-milman} for operators with operator valued kernels in order to obtain bounded toroidal pseudo-differential operators from $L^1(\mathbb{T}^n)$ into weak-$L^1(\mathbb{T}^n)$. On the torus, the validity of this weak (1,1) estimate was an open problem. Here we state the main result of this paper:
\begin{theorem}
    Let $T \in \Psi^m_{\rho, \delta}(\mathbb{T}^n \times \mathbb{Z}^n) $, let $0 < \rho \leq 1$, let $0 \leq \delta < 1$ and let $\lambda = \max\{ 0, (\delta-\rho)/2 \}$. Then:
    \begin{itemize}
        \item[(a)] if $m \leq - n [(1-\rho)/2 + \lambda]$, the operator $T$ extends to a bounded operator from the Hardy space $H^1(\mathbb{T}^n)$ into $L^1(\mathbb{T}^n)$, from $L^1(\mathbb{T}^n)$ into weak-$L^1(\mathbb{T}^n)$, and from $L^\infty(\mathbb{T}^n)$ into $\mathrm{BMO}(\mathbb{T}^n)$.\\
        \item[(b)] if $m \leq -n[(1-\rho)|1/p - 1/2| + \lambda]$, then $T$ extends to a bounded operator on $L^p(\mathbb{T}^n)$ for $1<p<\infty$.
    \end{itemize}    
    \label{theo:mainT}
\end{theorem}

\begin{remark}
    Now we discuss our main result. When $\delta < \rho$, we recover the order in Fefferman's estimate proved by Delgado \cite{delgado}, however, the estimate in \cite{delgado} considers symbols with limited regularity. Moreover, when $\delta \leq \rho$, in the case of the torus, we obtain the result proved in \cite{cardona-ruzhansky-subelliptic}. Since the Hardy space $H^1(\mathbb{T}^n)$ and the $\mathrm{BMO}(\mathbb{T}^n)$ space are not stable under multiplication by test functions, \cref{theo:mainT} cannot be obtained by  transferring the result by \'Alvarez and Hounie \cite{alvarez-hounie} to the torus using local partitions of unity. Finally, we remark that the Hörmander classes are stable under change of coordinates only when $\rho \geq  1 -\delta$. Therefore, the weak (1,1) estimate in \cref{theo:mainT} cannot be obtained by applying \'Alvarez and Hounie's result to the torus $\mathbb{T}^n,$ understood as a closed manifold, by the use of local coordinate systems and standard arguments using  partitions of unity.
\end{remark}
This  work is organised as follows. In \cref{section:prelims} we discuss the preliminaries about discrete Fourier analysis on the torus, as developed by Ruzhansky and Turunen \cite{ruzhansky-turunen}, and we define useful function spaces for this paper. In \cref{section:kernel-estimates} we prove estimates for the Shwartz kernel of toroidal pseudo-differential operators analogous to those proved in \cite{alvarez-hounie}. Finally, in \cref{section:boundedness} we prove our main results, particularly those stated in \cref{theo:mainT}.

\section{Preliminaries}
\label{section:prelims}
In this section we present the preliminaries about the Fourier analysis on the torus and on the toroidal pseudo-differential calculus as developed in \cite{ruzhansky-turunen}.

\begin{definition}[Periodic functions] 
    A function $f: \mathbb{R}^n \rightarrow X$, where $X$ is a Banach space, is 1-periodic if $f(x) = f(x + k)$ for every $ x \in \mathbb{R}^n $ and $ k \in \mathbb{Z}^n $. We can identify this function with one defined in $\mathbb{T}^n := \mathbb{R}^n / \mathbb{Z}^n $, which is a compact manifold without boundary. Moreover, the space of 1-perodic functions $m$ times continuously differentiable is denoted by $C^m(\mathbb{T}^n;X)$. The subset of these functions that have compact support is denoted by $C^m_0(\mathbb{T}^n;X)$  The test functions are elements of the space $C^\infty(\mathbb{T}^n;X) := \bigcap_m C^m(\mathbb{T}^n;X)$. When $X=\mathbb{C}$, we will simply use $C^\infty(\mathbb{T}^n)$.
\end{definition}
In order to define the class of pseudo-differential operators on $\mathbb{T}^n$, we first need to define the Fourier transform for 1-periodic smooth functions.
\begin{definition}[Fourier Transform on $\mathbb{T}^n$]
    For $f \in C^\infty(\mathbb{T}^n)$ we define the  \textit{toroidal Fourier transform} $\mathcal{F}_{\mathbb{T}^n}$ as 
    \begin{equation*}
        (\mathcal{F}_{\mathbb{T}^n}f)(\xi) := \int_{\mathbb{T}^n} e^{-i2\pi x \cdot \xi}f(x)\diff x, 
    \end{equation*}
    for $\xi \in \mathbb{Z}^n$.
\end{definition}
Thanks to the structure of the torus $\mathbb{T}^n$, the corresponding frequency domain of 1-periodic smooth functions is the lattice $\mathbb{Z}^n$. Also, it can be proved that the corresponding functions $\mathcal{F}_{\mathbb{T}^n}f$ satisfy certain regularity conditions. In fact, they belong to the Schwartz space defined below. 
\begin{definition}[Schwartz space $\mathcal{S}(\mathbb{Z}^n) $]
    We say $\varphi \in \mathcal{S}(\mathbb{Z}^n)$, that is a \textit{rapidly decaying} function $\mathbb{Z}^n \rightarrow \mathbb{C}$, if for given $0<M<\infty$, there exists $C_{\varphi M} > 0$ such that 

    \begin{equation*}
        |\varphi(\xi)| \leq C_{\varphi M} \langle\xi\rangle^{-M} \quad , \quad \forall\, \xi \in \mathbb{Z}^n.
    \end{equation*}
\end{definition}
It can be proved that the toroidal Fourier transform is a bijection between $C^\infty(\mathbb{T}^n)$ and $\mathcal{S}(\mathbb{Z}^n)$ with inverse  $\mathcal{F}_{\mathbb{T}^n}^{-1}:  \mathcal{S}(\mathbb{Z}^n) \rightarrow C^\infty(\mathbb{T}^n)$ defined by
    \begin{equation*}
        (\mathcal{F}_{\mathbb{T}^n}^{-1}\varphi)(x) := \sum_{\xi \in \mathbb{Z}^n} e^{i2\pi x \cdot \xi} \varphi(\xi).
    \end{equation*}
In the Euclidean case, the Hörmander symbol classes are defined requiring some regularity conditions using partial derivatives in both, the space and frequency domain. However, in the case of the torus the frequency domain is the lattice $\mathbb{Z}^n$. Therefore, we define the natural analogous of the derivative in the discrete case: the difference operator.
\begin{definition}[Partial difference operator]
    Let $e_j \in \mathbb{Z}^n$ such that $(e_j)_i$ is equal to 1 if $i=j$ and 0 otherwise. Then, for $p(\xi):\mathbb{Z}^n \rightarrow \mathbb{C} $ we define 
    \begin{equation*}
        \Delta_{\xi_j}p(\xi) := p(\xi + e_j) - p(\xi) \quad \text{and} \quad \Delta^\alpha_\xi := \Delta_{\xi_1}^{\alpha_1} \cdots \Delta_{\xi_n}^{\alpha_n},
    \end{equation*}
    
    for any multi-index $\alpha \in \mathbb{N}^n_0$
\end{definition}
Now we proceed to define the Hörmander classes of symbols on $ \mathbb{T}^n \times \mathbb{Z}^n $, as in Ruzhansky and Turunen \cite{ruzhansky-turunen}, that will be important in the quantization of pseudo-differential operators.
\begin{definition}[Hörmander classes on $ \mathbb{T}^n \times \mathbb{Z}^n $]
    Let $m \in \mathbb{R} $ and $0 \leq \delta, \rho\leq 1$. We say that a function $p:=p(x, \xi)$ that is smooth on $x$ for any $\xi \in \mathbb{Z}^n$ belongs to the \textit{toroidal symbol class} $S^m_{\rho,\delta} (\mathbb{T}^n\times \mathbb{Z}^n) $ if 

    \begin{equation*}
        \left|\partial^\beta_x \Delta^\alpha_\xi p(x, \xi)\right| \leq C_{\alpha\beta}\langle\xi\rangle^{m-\rho|\alpha| + \delta|\beta|},
    \end{equation*}
    
    for every $x \in \mathbb{T}^n$, $\alpha, \beta \in \mathbb{N}^n_0$ and $\xi \in \mathbb{R}^n$. Moreover, we say that $p(x, \xi)$ has order $m$ and we define $S^{-\infty}_{\rho,\delta} (\mathbb{T}^n\times \mathbb{Z}^n) = \bigcap_{m\in\mathbb{R}}S^m_{\rho,\delta} (\mathbb{T}^n\times \mathbb{Z}^n) $.
\end{definition}

\begin{definition}[Pseudo-differential operators on $ \mathbb{T}^n \times \mathbb{Z}^n $]
    For a symbol $p:=p(x, \xi) \in S^m_{\rho,\delta} (\mathbb{T}^n\times \mathbb{Z}^n) $ we can associate a \textit{toroidal pseudo-differential operator} $\mathrm{Op}(p)$ from $C^\infty(\mathbb{T}^n)$ into itself, defined as  
    \begin{equation*}
        \mathrm{Op}(p)f(x):=\sum_{\xi \in \mathbb{Z}^n} e^{i2\pi x \cdot \xi}p(x, \xi)(\mathcal{F}_{\mathbb{T}^n}f)(\xi),
    \end{equation*}
    which can be rewritten as 
    \begin{equation}
        \mathrm{Op}(p)f(x)=  \sum_{\xi \in \mathbb{Z}^n} \int_{\mathbb{T}^n} e^{i2\pi (x-y) \cdot \xi} p(x, \xi)f(y) \diff y .
        \label{eq:pdo-def}
    \end{equation}
    The class of operators with symbols on $S^m_{\rho,\delta} (\mathbb{T}^n\times \mathbb{Z}^n) $ will be denoted  by $\Psi^m_{\rho,\delta} (\mathbb{T}^n\times \mathbb{Z}^n) $.
\end{definition}
    The toroidal Fourier transform and toroidal pseudo-differential operators can be extended by duality to the space of \textit{periodic distributions} $\mathcal{D}'(\mathbb{T}^n)$ consisting of continuous linear functionals on $C^\infty(\mathbb{T}^n)$. This extension allows us to define the Schwartz kernel of a toroidal pseudo-differential operator.
\begin{remark}[Schwartz kernel]
    In the sense of distributions described above, \cref{eq:pdo-def} can be rewritten as
    \begin{equation*}
         \mathrm{Op}(p)f(x)=\int_{\mathbb{T}^n}\left[\sum_{\xi \in \mathbb{Z}^n} e^{i2\pi(x - y) \cdot \xi} p(x, \xi)\right]f(y)\diff y = \int_{\mathbb{T}^n}k(x, y)f(y)\diff y,
    \end{equation*}
    and we say $k(x, y)$ is the \textit{Schwartz kernel} of the corresponding operator. Notice that when the order of $p(x, \xi)$ is less than $-n$, the series defining $k$ is absolutely convergent and the kernel is a well-defined function.
\end{remark}
Since the proofs of the results in this paper only depend on the order of the toroidal pseudo-differential operators, then they extend to their adjoints. This is justified by the following result by \cite[Corollary~4.9.8]{ruzhansky-turunen}.
\begin{theorem}
    If $T \in \Psi^m_{\rho,\delta} (\mathbb{T}^n \times \mathbb{Z}^n) $, then its adjoint $T^* \in \Psi^m_{\rho,\delta} (\mathbb{T}^n \times \mathbb{Z}^n) $.
\end{theorem}
Now we introduce a particular and very useful toroidal pseudo-differential operator.
\begin{definition}[Bessel's potential]
    We define \textit{Bessel's potential} $J^s$ as the pseudo-differential operator with symbol $\langle\xi\rangle^s$ for $s \in \mathbb{R}$.
\end{definition}

For our pourposes, we will use a smooth interpolation of a symbol on $\mathbb{T}^n \times \mathbb{Z}^n $ to obtain a symbol defined on $\mathbb{T}^n \times \mathbb{R}^n $. This is possible in view of the results stated in \cite[section~4.5]{ruzhansky-turunen}, here we write some useful consequences.

\begin{theorem}
    Let $m \in \mathbb{R}$, $0\leq\delta<1$, $0< \rho \leq 1$. The symbol $p \in S^m_{\rho,\delta} (\mathbb{T}^n\times \mathbb{Z}^n) $ is a toroidal symbol if and only if there exists a symbol $\tilde{p} \in S^m_{\rho,\delta} (\mathbb{T}^n\times \mathbb{R}^n) $ such that $p = \tilde{p}|_{\mathbb{T}^n\times \mathbb{Z}^n} $. Moreover, this extension is unique modulo $S^{-\infty}(\mathbb{T}^n\times \mathbb{R}^n)$.
\end{theorem}

\begin{remark}
    When we employ this extension, the definition of its corresponding operator may be adjusted to 
    \begin{equation*}
        \mathrm{Op}(\tilde{p})f(x) = \int_{\mathbb{T}^n} \left[ \int_{\mathbb{R}^n} e^{i2\pi (x-y) \cdot \xi} \tilde{p}(x, \xi)\diff \xi \right] f(y) \diff y ,
    \end{equation*}
    so the Schwartz kernel of the operator, in the sense of distributions, is defined as the integral 
    \begin{equation*}
        \tilde{k}(x, y) = \int_{\mathbb{R}^n} e^{i2\pi (x-y) \cdot \xi} \tilde{p}(x, \xi)\diff \xi.
    \end{equation*}
    This kernel representation allows us to use techniques such as integration by parts when discussing properties of the Schwartz kernel.
\end{remark}
Hence, there is a correspondence between toroidal symbols with discrete and continuous frequency domains. This translates to the corresponding operators as well.
\begin{theorem}[Equivalence of operator classes]
    Let $m \in \mathbb{R}$, $0\leq\delta<1$, $0< \rho \leq 1$. Then 
        \begin{equation*}
            \Psi^m_{\rho,\delta}(\mathbb{T}^n\times \mathbb{R}^n) = \Psi^m_{\rho,\delta}(\mathbb{T}^n\times \mathbb{Z}^n).
        \end{equation*}
\end{theorem}
Now, we define some function spaces on the torus taking values on a Banach space $X$, used to prove our boudedness results.

\begin{definition}[Bochner-Lebesgue spaces]
    \cite[pp.~9]{defrancia} For a Banach space $X$ denote $L^p(\mathbb{T}^n;X)$, the \textit{Bochner-Lebesgue space}, consisting of strongly measureable $X$-valued functions $f$ defined on $\mathbb{T}^n$ such that 
    \begin{align*}
        \|f\|_{L^p(\mathbb{T}^n;X)} := \left( \int_{\mathbb{T}^n} \|f(x)\|_X^p \diff x \right)^{1/p} < \infty \; &, \quad 1\leq p < \infty,\\ 
        \|f\|_{L^\infty(\mathbb{T}^n;X)} := \esssup_{x \in \mathbb{T}^n} \|f(x)\|_X  < \infty. &
    \end{align*}
    Similarly, we can define weak-$L^p(\mathbb{T}^n;X)$ as those functions satisfying
    \begin{equation*}
        \|f\|_{weak-L^p(\mathbb{T}^n;X)} := \sup_{\lambda > 0} \lambda|\{x \in \mathbb{T}^n : \|f(x)\|_X > \lambda \}|^{1/p} < \infty.
    \end{equation*}
    When $X = \mathbb{C}$, we simply refer to these spaces as $L^p(\mathbb{T}^n).$ By simplifying the notation, we also use $\|\cdot\|_p$ for the $L^p$-norm.
    \label{def:bochner-lebesgue}
\end{definition}
We recall the definition of bounded linear operators between Banach spaces in order to define operators with operator valued kernel.
\begin{definition}[Bounded operators between Banach spaces $\mathcal{B}(X,Y)$]
    We say $T:X\rightarrow Y$ is a bounded linear operator between Banach spaces, namely $T \in \mathcal{B}(X, Y)$, if 
    \begin{equation*}
        \|Tx\|_Y \leq C\|x\|_X  \quad \text{for every } x\in X.
    \end{equation*}
    Also, the norm is defined as 
    \begin{equation*}
        \|T\|_{\mathcal{B}(X,Y)} := \inf \left\{ C \in \mathbb{R}^+ :  \|Tx\|_Y \leq C\|x\|_X \quad \text{for every } x\in X \right\}.
    \end{equation*}
\end{definition}
In order to state a theorem analogous to the Euclidean case proved in \cite[Theorem~4.1]{alvarez-milman}, we define linear operators with operator valued kernel.
\begin{definition}[Operator valued kernel]
\cite[pp.~29]{defrancia} We say that an operator $T:C^\infty(\mathbb{T}^n;X)\rightarrow C^\infty(\mathbb{T}^n;Y)$ has an operator valued kernel if it can be written as 
\begin{equation*}
    Tf(x) = \int_{\mathbb{T}^n}k(x, y)f(y)\diff y,
\end{equation*}
 where $k:\mathbb{T}^n \times \mathbb{T}^n \rightarrow \mathcal{B}(X,Y)$, called the kernel, is such that $\|k(x, \cdot)\|_{\mathcal{B}(X,Y)}$ is integrable away of $x \in \mathbb{T}^n $.
    \label{def:operator-kernel}
\end{definition}
It has been shown that toroidal pseudo-differential operators in the Hörmander class of order zero are not bounded on $L^p(\mathbb{R}^n)$, in general for $p=1,\infty$, see e.g. the Ph.D. thesis of Wang \cite{wang}. This counterexample can be extended to the torus. Hence, it is required to define adequate subspaces from which we can interpolate to $L^p(\mathbb{T}^n)$ spaces for $1<p<\infty$. Those spaces happen to be the Hardy and BMO spaces, see \cite[pp.~159]{fefferman-stein}.
\begin{definition}[Hardy space $H^1(\mathbb{T}^n)$]
    We say that $a$ is an $H^1(\mathbb{T}^n)$\textit{-atom} if $a$ is supported on a ball $B \subset \mathbb{T}^n$ so that 
    \begin{equation*}
        \|a\|_{L^\infty(\mathbb{T}^n)} \leq |B|^{-1} \quad \text{ and } \quad \int_Ba(x)\diff x = 0.
    \end{equation*}
    Also, we say that $f \in H^1(\mathbb{T}^n)$ if it can be written as  
    \begin{equation*}
        f = \sum_{j \in \mathbb{Z}^+}\lambda_ja_j \;\; \text{ where } \;\; \sum_{j \in \mathbb{Z}^+}|\lambda_j| < \infty,
    \end{equation*}
    and $a_j$ are $H^1(\mathbb{T}^n)$-atoms. An expression of this form is called an \textit{atomic decomposition of }$f$. Also, we define the $H^1(\mathbb{T}^n)$\textit{-norm} as 
    \begin{equation*}
        \|f\|_{H^1(\mathbb{T}^n)} := \inf \sum_{j \in \mathbb{Z}^+}|\lambda_j|,
    \end{equation*}
    where the infimum is taken over all atomic decompositions of $f$. 
\end{definition}

\begin{definition}[Bounded Mean Oscillation space $\mathrm{BMO}(\mathbb{T}^n)$]
    We say $f \in \mathrm{BMO}(\mathbb{T}^n)$, namely, $f$ has \textit{bounded mean oscillation}, if 
    \begin{equation*}
        \|f\|_{\mathrm{BMO}(\mathbb{T}^n)}' := \sup_{B\subset\mathbb{T}^n} \frac{1}{|B|} \int_B \left| f(x) - f_B \right| \diff x < \infty, \quad \mathrm{ where } \quad f_B 
        := \frac{1}{|B|}\int_B f(x)\diff x,
    \end{equation*}
    and $B$ are balls contained on $\mathbb{T}^n$. Here we will use an equivalent norm to $\|\cdot\|_{\mathrm{BMO}(\mathbb{T}^n)}'$ defined as 
    \begin{equation*}
        \|f\|_{\mathrm{BMO}(\mathbb{T}^n)} := \sup_{B\subset\mathbb{T}^n} \inf_{b\in \mathbb{C}} \frac{1}{|B|} \int_B \left| f(x) - b \right| \diff x.
    \end{equation*}
\end{definition}

In \cite{fefferman-BMO}, Charles Fefferman proved the famous result that $\mathrm{BMO}(\mathbb{R}^n)$ is the dual space of the Hardy space $H^1(\mathbb{R}^n)$. This result was extended to homogeneous groups, particularly the torus, see Folland and Stein \cite{folland-stein}. This duality can be understood on the following sense:
\begin{itemize}
    \item[(a)] For any $\phi \in \mathrm{BMO}(\mathbb{T}^n)$ we can define the functional $f\mapsto \int_{\mathbb{T}^n}f(x)\phi(x)\diff x$ which extends to a bounded functional on $H^1(\mathbb{T}^n)$.
    \item[(b)] Conversely, any continuous functional on $H^1(\mathbb{T}^n)$ can be identified with a functional defined as in (a) for a unique function $\phi \in \mathrm{BMO}(\mathbb{T}^n)$.
\end{itemize}

Now, we state a theorem that will be useful in proving some estimates in the next section. Notice this result by Hounie \cite{hounie} is valid for Euclidean pseudo-differential operators. The toroidal analogous will be stated in the next section. 
\begin{theorem}
    Let $\tilde{p}:\mathbb{R}^n \times \mathbb{R}^n \rightarrow \mathbb{C}$ be a symbol such that for $0<\rho\leq1$, $0\leq \delta < 1$, $m\leq -n \lambda = -n\max\{0, (\delta - \rho)/2\}$ and $|\alpha|,|\beta| \leq \lceil n/2\rceil$ satisfies: 
    \begin{equation}
        \left|\partial^\alpha_\xi \partial^\beta_x \tilde{p}(x, \xi)\right| \leq C_{\alpha\beta}\langle\xi\rangle^{m-\rho|\alpha| + \delta|\beta|}
    \end{equation}

    Then $\mathrm{Op}(\tilde{p})$ is bounded from $L^2(\mathbb{R}^n)$ into $L^2(\mathbb{R}^n)$ with norm proportional to the best bound $C_{\alpha\beta}$.
    \label{theo:22boundR}
\end{theorem}

\section{Main results}
\label{section:main}
In this section we prove the main results of this paper. These are presented into two subsections.
In the first Subsection \ref{section:kernel-estimates} we prove analogous estimates to those proved in \cite[Sections~1 and 2]{alvarez-hounie} for Hörmander classes of toroidal pseudo-differential operators. In the second Subsection \ref{section:boundedness} we prove the toroidal analogous of the $L^1$-weak-$L^1$ boundedness result by \'Alvarez and Milman \cite[Theorem 4.1]{alvarez-milman} for a certain class of operators with operator valued kernel, in order to obtain  $L^1(\mathbb{T}^n)$-weak-$L^1(\mathbb{T}^n)$ boundedness criteria for toroidal pseudo-differential operators. Moreover, we prove $H^1(\mathbb{T}^n)$-$L^1(\mathbb{T}^n)$ and $L^\infty(\mathbb{T}^n)$-$\mathrm{BMO}(\mathbb{T}^n)$ boundedness criteria for toroidal pseudo-differential operators. Then we proceed with the Fefferman-Stein complex interpolation argument to obtain the $L^p(\mathbb{T}^n)$ boundedness for these classes of operators.

\subsection{Kernel estimates on the torus}
\label{section:kernel-estimates}
Here we prove useful estimates for the Schwartz kernel of toroidal pseudo-differential operators. The following theorem is the toroidal analogous of the one proved in \cite[Theorem~1.1]{alvarez-hounie}.
\begin{theorem}
    Let $T \in \Psi^m_{\rho, \delta}(\mathbb{T}^n \times \mathbb{Z}^n) $, $0 < \rho \leq 1$, $0 \leq \delta < 1$, with symbol $p(x, \xi)$ and with kernel

    \begin{equation}
        k(x, y) = \sum_{\xi \in \mathbb{Z}^n} e^{i2\pi(x - y) \cdot \xi} p(x, \xi).
    \end{equation}
    \begin{enumerate}
        \item[(a)] (Pseudo-local property) $k$ is smooth outside the diagonal. Moreover, given $\alpha, \beta \in \mathbb{N}^n_0$, then for any $N > (m + n + |\alpha + \beta|)/\rho$ we have
        \begin{equation}
            \sup_{x \neq y} |x - y|^N |\partial^\alpha_x \partial^\beta_y k(x, y)| = C_{\alpha\beta N} < \infty.
            \label{eq:teo3-1a}
        \end{equation}

        \item[(b)] Assuming  $p$ has compact support in $\xi$ uniformly in $x$, then $k$ is smooth and given $\alpha, \beta \in \mathbb{N}^n_0$, there is $C>0$ such that 
        \begin{equation}
            |\partial^\alpha_x \partial^\beta_y k(x, y)| \leq C\langle x - y \rangle^{-N}.
        \end{equation}

        \item[(c)] Assuming $m + M + n < 0$ for some $M \in \mathbb{Z}^+$, then $k$ is a bounded continuous function with bounded continuous derivatives up to order $M$.\\

        \item[(d)] Assuming $m + M + n = 0$ for some $M \in \mathbb{Z}^+$, then there is $ C >0$ such that 
        \begin{equation}
            \sup_{|\alpha + \beta| = M} |\partial^\alpha_x \partial^\beta_y k(x, y)| \leq C | \log |x-y| | \; , \quad x \neq y.
        \end{equation} 
    \end{enumerate}
\end{theorem}
\begin{proof}
    First, we notice that there is $\tilde{p} \in S^m_{\rho, \delta}(\mathbb{T}^n \times \mathbb{R}^n)$ such that $\tilde{p}|_{\mathbb{T}^n \times \mathbb{Z}^n} = p$ and with $\mathrm{Op}(\tilde{p}) = T$. In consequence, $k(x, y) = \int_{\mathbb{R}^n} e^{i2\pi(x-y)\cdot \xi} \tilde{p}(x, \xi)\diff\xi $. Now, the derivatives of the kernel look as follows:
        \begin{equation}
            \partial^\alpha_x \partial^\beta_y k(x, y) = \int_{\mathbb{R}^n} (-i2\pi\xi)^\beta e^{i2\pi(x-y)\cdot \xi} \sum_{\omega \leq \alpha} C_\omega (i2\pi\xi)^{\alpha - \omega} \partial^\omega_x \tilde{p}(x, \xi) \diff\xi,
            \label{eq:derivative-kernel}
        \end{equation}
    which is the kernel of an operator with symbol of order $m + |\alpha + \beta|$. Then, it is sufficient to prove the results when $|\alpha + \beta| = 0$.

    \begin{itemize}
        \item[(a)] The continuity of the kernel $k(x, y)$ is proven in \cite[Theorem~4.3.6]{ruzhansky-turunen}. By integration by parts we have that
        \begin{align*}
            (i2\pi)^{|\gamma|}(x - y)^\gamma k(x, y) =& \int_{\mathbb{R}^n} \partial^\gamma_\xi \left[e^{i2\pi(x-y)\cdot \xi} \right]  \tilde{p}(x, \xi) \diff\xi \\
            = & (-1)^{|\gamma|} \int_{\mathbb{R}^n} e^{i2\pi(x-y)\cdot \xi} \partial^\gamma_\xi \left[\tilde{p}(x, \xi) \right] \diff\xi .
        \end{align*}

        Hence, if we set $|\gamma| = N$, we get
        \begin{align*}
            |i2\pi|^N|x-y|^N| k(x, y)| \leq & \int_{\mathbb{R}^n} \langle\xi\rangle^{m-\rho N} \diff \xi.
        \end{align*}

        The last integral is finite when $N > (m + n)/\rho$, proving the result.\\

        \item[(b)] We notice that the kernel would be a finite sum of continuous functions, proving the  continuity of $k(x, y)$. Also, $\tilde{p}$ would have the same support of $p$. Hence the last integral above would be finite without any restrictions on $N$.\\

        \item[(c)] Let $ m < -n$. Then, we have the finite series
        \begin{align*}
            | k(x, y)| \leq &  \sum_{\xi \in \mathbb{Z}^n}\langle\xi\rangle^m,
        \end{align*}
        proving the boundedness of $k(x, y)$.\\

        \item[(d)] First notice that by (a), (b) it suffices to prove the estimate when $|x-y|<1$ and if $\tilde{p}(x, \xi)$ vanishes for $|\xi|<1$ uniformly on $x$. Let $m + n = 0$ and let $\varphi \in C_0^\infty(\mathbb{R})$ be supported on $[0, 1]$, such that $\int \varphi=1$, and let us define 
        \begin{equation*}
            k(x, y, t) = \int_{\mathbb{R}^n} e^{i2\pi(x-y)\cdot \xi} \tilde{p}(x, \xi) \varphi(\langle\xi\rangle-t) \diff\xi,
        \end{equation*}
        so that
        \begin{equation*}
            k(x, y) = \int_1^\infty k(x, y, t)\diff t.
        \end{equation*}
        Then, using integration by parts, we have that
        \begin{align*}
            (i2\pi)^{|\gamma|}(x - y)^\gamma k(x, y, t) = & \int_{\mathbb{R}^n} \partial^\gamma_\xi\left[ e^{i2\pi(x-y)\cdot \xi} \right] \tilde{p}(x, \xi)\varphi(\langle\xi\rangle - t) \diff \xi \\
            = & (-1)^{|\gamma|} \int_{\mathbb{R}^n} e^{i2\pi(x-y)\cdot \xi} \sum_{\omega \leq \gamma}C_\omega \partial^\omega_\xi \tilde{p}(x, \xi) \partial^{\gamma - \omega}\varphi(\langle\xi\rangle-t) \diff \xi.
        \end{align*}
        Since $\langle\xi\rangle \sim t$ in the support of $\varphi(\cdot-t)$ which has volume estimated by $Ct^{n-1}$, we get for $|\gamma|=N$
        \begin{align*}
            |x-y|^N|k(x, y, t)| \leq & \;C \int_{\supp \varphi(\cdot - t)} \langle\xi\rangle^{m - \rho N} \diff \xi \\
            \leq & \; C\int_{\supp \varphi(\cdot - t)} t^{m-\rho N}\diff\xi \leq Ct^{m+n - \rho N-1}.
        \end{align*}
        By adding the estimates for both $N = 0, 1$, we get 
        \begin{equation*}
            |k(x, y, t)| \leq C \frac{t^{ -\rho - 1}}{t^{-\rho } + |x - y|}.
        \end{equation*}           
    \end{itemize} Then, the result of evaluating the integral above is the desired estimate $|k(x, y)| \leq  C|\log |x-y|| $ for $x \neq y$.
\end{proof}
\begin{remark}
    In the following theorem, we prove the toroidal counterparts of estimates proved in \cite[Theorem~2.1]{alvarez-hounie}. Nevertheless, for the estimates in (a) we are able to drop restrictions on the order of the operator thanks to the compactness of $\mathbb{T}^n$.
\end{remark}
\begin{remark}
    The estimates below for $\sigma \geq 1$ remain valid when $\sigma \geq \varepsilon$ where $\varepsilon>0$ with the constants $C=C_\varepsilon$ depending on $\varepsilon$.
\end{remark}
\begin{theorem}
    Let $T \in \Psi^m_{\rho, \delta}(\mathbb{T}^n \times \mathbb{Z}^n) $, $0 < \rho \leq 1$, $0 \leq \delta < 1$, with symbol $p:=p(x, \xi)$ and with kernel $k:=k(x, y)$. Set $\lambda = \max\{ 0, (\delta - \rho)/2 \}$.
    \begin{itemize}
        \item[(a)] For any fixed $z \in \mathbb{T}^n$, and $\sigma \geq 1$, we have the kernel inequalities:
        \begin{equation}
            \sup_{|y-z| \leq \sigma} \int_{|x-z| > 2\sigma}|k(x, y) - k(x, z)|\diff x \leq C,
        \end{equation}
        \begin{equation}
            \sup_{|y-z| \leq \sigma} \int_{|x-z| > 2\sigma}|k(y, x) - k(z, x)|\diff x \leq C.
        \end{equation}
        \item[(b)] If $m \leq -n [(1-\rho)/2 + \lambda] $ and $\sigma < 1$, we have for any fixed $z \in \mathbb{T}^n$, 
        \begin{equation}
            \sup_{|y-z| \leq \sigma} \int_{|x-z| > 2\sigma^\rho}|k(x, y) - k(x, z)|\diff x \leq C.
        \end{equation}
        \item[(c)] If $m \leq -n (1 - \rho) / 2$ and $\sigma < 1$, we have  for any fixed $z \in \mathbb{T}^n$,
        \begin{equation}
            \sup_{|y-z| \leq \sigma} \int_{|x-z| > 2\sigma^\rho}|k(y, x) - k(z, x)|\diff x \leq C.
        \end{equation} 
    \end{itemize}
    \label{theo:sigma-kernel-estimate}
\end{theorem}

\begin{proof}
    \begin{itemize}
        \item[(a)] First, we notice that $|x-y| \geq |x - z| - |z - y| > \sigma$ in the domain of evaluation. Then by the triangle inequality and \cref{eq:teo3-1a} we have
        \begin{equation*}
            \sup_{|y-z| \leq \sigma} \int_{|x-z| > 2\sigma}|k(x, y) - k(x, z)|\diff x  
        \end{equation*}
        \begin{equation*}
            \leq \sup_{|y-z| \leq \sigma} \int_{|x-z| > 2\sigma}|k(x, y)|\diff x + \sup_{|y-z| \leq \sigma} \int_{|x-z| > 2\sigma}|k(x, z)|\diff x
        \end{equation*}
        \begin{equation*}
            \leq  \int_{|x-y|>\sigma} C_N|x-y|^{-N} \diff x + \int_{|x-z|>\sigma} C_N|x-z|^{-N} \diff x
        \end{equation*}
        \begin{equation*}
            \leq  \int_{\mathbb{T}^n} C_N\sigma^{-N} \diff x + \int_{\mathbb{T}^n} C_N\sigma^{-N} \diff x \leq C.
        \end{equation*}

        \item[(b)] As before, let $\tilde{p}$ be the corresponding symbol on $\mathbb{T}^n \times \mathbb{R}^n$.  Let $\varphi \in C_0^\infty(\mathbb{R})$ supported in $[1/2, 1]$ such that 
        \begin{equation*}
            \int_0^\infty\varphi(1/t)/t \diff t =\int_1^2\varphi(1/t)/t \diff t = 1.
        \end{equation*} 
        Define 
        \begin{equation*}
            k(x, y, t) = \int_{\mathbb{R}^n} e^{i2\pi(x - y) \cdot \xi} \tilde{p}(x, \xi) \varphi(\langle\xi\rangle/t) \diff \xi,
        \end{equation*}
        so that  
        \begin{equation*}
            k(x, y) = \int_0^\infty k(x, y, t) \diff t = \int_1^\infty k(x, y, t) \diff t.
        \end{equation*}
        Let $N > n/2$ be an integer, then we obtain the estimates
        \begin{equation*}
            \int_{|x-z|>2\sigma^\rho} |k(x, y, t) - k(x, z, t)| \diff x
        \end{equation*}
        \begin{equation*}
            \leq  \left[ \int_{\mathbb{T}^n} (1+t^{2\rho}|x-z|^2)^N |k(x, y, t) - k(x, z, t)| \diff x  \right]^{1/2} \left[ \int_{\mathbb{T}^n} (1+t^{2\rho}|x-z|^2)^{-N} \diff x \right]^{1/2} 
        \end{equation*}
        \begin{equation}
            \leq C\left[ \int_{\mathbb{T}^n} (1+t^{2\rho}|x-z|^2)^N |k(x, y, t) - k(x, z, t)| \diff x  \right]^{1/2}t^{-\rho n/2}.
            \label{eq:firstbound-b}
        \end{equation}
        So, for $|\alpha| \leq N$ we have that
        \begin{equation*}
            t^{\rho|\alpha|}(x-z)^\alpha \int_{\mathbb{R}^n} \left[ e^{i2\pi(x - y) \cdot \xi} - e^{i2\pi(x - z) \cdot \xi} \right]\tilde{p}(x, \xi)\varphi(\langle\xi\rangle/t)\diff \xi 
        \end{equation*}
        \begin{equation*}
            =t^{\rho|\alpha|}(x-z)^\alpha \int_{\mathbb{R}^n} e^{i2\pi(x - z) \cdot \xi} \left[ e^{i2\pi(z - y) \cdot \xi} - 1 \right]\tilde{p}(x, \xi)\varphi(\langle\xi\rangle/t)\diff \xi 
        \end{equation*}
        \begin{equation*}
            =\sum_{\beta\leq\alpha} C_{\alpha\beta} t^{\rho|\alpha|} \int_{\mathbb{R}^n} e^{i2\pi(x - z) \cdot \xi} \partial^\beta_\xi \left[ 
            \left(e^{i2\pi(z - y) \cdot \xi} - 1 \right) \tilde{p}(x, \xi) \right] \partial^{\alpha-\beta}_\xi \varphi(\langle\xi\rangle/t)\diff \xi.
        \end{equation*}
        Now, $|e^{i2\pi(x - z) \cdot \xi}-1| \leq |x-z||\xi|\leq t\sigma$ in the support of $\varphi(\langle\xi\rangle/t)$. On the other hand, $|\partial^\gamma_\xi e^{i2\pi(x - z) \cdot \xi}| \leq C_\gamma|y-z|^{|\gamma|} \leq C_\gamma \sigma ^{|\gamma|} \leq C\langle\xi\rangle^{-|\gamma|}(t\sigma)^{|\gamma|} $. Let us assume that $t\sigma < 1$, then for any $\chi \in C_0^\infty(\mathbb{R}^n)$ that is equal to $1$ on the support of $\varphi(\langle\xi\rangle/t)$ the set
        \begin{equation*}
            \Sigma_{\alpha\beta}=\left\{ \langle\xi\rangle^{n(1 - \rho)/2 + \rho|\beta|} \partial^\beta_\xi \left[\left( e^{i2\pi(x - z) \cdot \xi}-1 \right) 
            \tilde{p}(x, \xi) \right] \chi(\langle\xi\rangle/t) : |y-z| < \sigma,\; z \in \mathbb{T}^n \right\}
        \end{equation*}
        has measure bounded by $C\langle\xi\rangle^{n(1 - \rho)/2 + \rho|\beta|} (t\sigma)\langle\xi\rangle^{m-\rho |\beta|} \leq Ct\sigma \langle\xi\rangle^{-n\lambda} $. Thus, we can consider $\tilde{p}(x, \xi)$ as a symbol on $\mathbb{R}^n\times\mathbb{R}^n$ and by \cref{theo:22boundR} each of the corresponding operators with symbols $$\langle\xi\rangle^{n(1 - \rho)/2 + \rho|\beta|} \partial^\beta_\xi \left[\left( e^{i2\pi(x - z) \cdot \xi}-1 \right) 
            \tilde{p}(x, \xi) \right] \chi(\langle\xi\rangle/t),$$ on the set $\Sigma_{\alpha\beta}$ are bounded on $L^2(\mathbb{R}^n)$ with norm estimated by $Ct\sigma$.
        Hence \cref{eq:firstbound-b} can be estimated using Plancherel's identity by 
        \begin{equation}
            Ct\sigma\sum_{\beta\leq\alpha,\;|\alpha|\leq N} C_{\alpha\beta}t^{\rho|\alpha|}\left\|  
             \langle\xi\rangle^{-n(1-\rho)/2 - \rho|\beta|} t^{-|\alpha-\beta|} \partial^{\alpha-\beta}_\xi \varphi(\langle\xi\rangle/t) \right\|_{L^2(\mathbb{R}^n)}t^{-\rho n/2} 
        \end{equation}
        \begin{equation}
            \leq Ct^{\rho n/2}t^{-\rho n/2}t\sigma =Ct\sigma.
            \label{eq:tsigma<1-b}
        \end{equation}
        Now, let us dropp the restriction $t\sigma < 1$. For $|\alpha| = N$ we have that
        \begin{equation*}
            \int_{|x-z|>2\sigma^\rho} |k(x, y, t)| \diff x 
        \end{equation*}
        \begin{equation*}
            \leq \left[ \int_{\mathbb{T}^n} \left( t^{2\rho} |x-y|^2 \right)^N |k(x, y, t)|^2 \diff x \right]^{1/2} \left[ \int_{|x-y|>\sigma^\rho} \left(|t^{2\rho}|x-y|^2\right)^{-N}\diff x \right]^{1/2}
        \end{equation*}
        \begin{equation}
            \leq C \left[ \int_{\mathbb{T}^n} \left( t^{2\rho}|x-y|^2 \right)^N|k(x,y,t)|^2 \diff x \right]^{1/2}t^{-\rho N}\sigma^{\rho(n/2 - N)}.
            \label{eq:secondbound-b}
        \end{equation} 
        Let $|\alpha|=N$, then 
        \begin{equation*}
            t^{\rho|\alpha|} (x-y)^\alpha \int_{\mathbb{R}^n} e^{i2\pi(x - y) \cdot \xi} \tilde{p}(x, \xi)\varphi(\langle\xi\rangle/t)\diff\xi 
        \end{equation*}
        \begin{equation*}
            = \sum_{\beta\leq\alpha} C_{\alpha\beta}t^{\rho|\alpha|} \int_{\mathbb{R}^n} e^{i2\pi(x - y) \cdot \xi} \partial^\beta_\xi \tilde{p}(x, \xi)t^{-|\alpha-\beta|} \partial^{\alpha-\beta}_\xi \varphi(\langle\xi\rangle/t)\diff \xi,
        \end{equation*}
        and for each $\beta\leq\alpha$ the $L^2(\mathbb{T}^n)$ norm as a $x$-function of 
        \begin{equation*}
            \int_{\mathbb{R}^n} e^{i2\pi(x - y) \cdot \xi} \partial^\beta_\xi \tilde{p}(x, \xi)t^{-|\alpha-\beta|} \partial^{\alpha-\beta}_\xi \varphi(\langle\xi\rangle/t)\diff \xi
        \end{equation*}
        is equal to the $L^2(\mathbb{T}^n)$ norm as a $x$-function of 
        \begin{equation*}
            \int_{\mathbb{R}^n} e^{i2\pi x \cdot \xi} \partial^\beta_\xi \tilde{p}(x+y, \xi)t^{-|\alpha-\beta|} \partial^{\alpha-\beta}_\xi \varphi(\langle\xi\rangle/t)\diff \xi.
        \end{equation*}
        On the other hand, the set $\{ \langle\xi\rangle^{n(1-\rho)/2 + \rho|\beta|} \partial^\beta_\xi\tilde{p}(x + y, \xi) : y \in \mathbb{T}^n  \}$ has measure bounded by $C\langle\xi\rangle^{n(1-\rho)/2 + \rho|\beta|} \langle\xi\rangle^{m- \rho|\beta|} = C\langle\xi\rangle^{n(1-\rho)/2 +m}   $ so their respective operators are bounded on $L^2(\mathbb{R}^n)$ by \cref{theo:22boundR}. Thus, \cref{eq:secondbound-b} can be estimated by  

        \begin{equation*}
            C\sum_{\beta\leq\alpha,\; |\alpha|=N} t^{\rho|\alpha|}t^{-n(1-\rho)/2 - \rho|\beta|}t^{-\rho|\alpha-\beta|}t^{n/2}t^{\rho N} \sigma^{\rho(n/2 - N)} \leq C(t\sigma)^{\rho (n/2 - N)}.
        \end{equation*}
        In a similar way, we can estimate 
        \begin{equation*}
            \int_{|x-z|>2\sigma^\rho} |k(x, z, t)| \diff x \leq C(t\sigma)^{\rho(n/2 - N)}.
        \end{equation*}
        Using these estimates and \cref{eq:tsigma<1-b} we get the result from the following expression:
        \begin{equation}
            \int_{|x-z| > 2\sigma^\rho}|k(x, y) - k(x, z)|\diff x \leq C\left[ \int_1^{1/\sigma} t\sigma + \int_{1/\sigma}^\infty(t\sigma)^{\rho(n/2-N)} \right]\diff t/t \leq C
            \label{eq:lastbound-b}.
        \end{equation} 
        Thus, completing the proof.\\

        \item[(c)] First, let us notice that
        \begin{equation*}
            k(y, x, t) - k(z, x, t) 
        \end{equation*}
        \begin{equation*}
            =\int_{\mathbb{R}^n} e^{-i2\pi (x-y) \cdot \xi}[\tilde{p}(y, \xi) - \tilde{p}(z, \xi)]\varphi(\langle\xi\rangle/t) \diff \xi
        \end{equation*}
        \begin{equation*}
            +  
            \int_{\mathbb{R}^n} e^{-i2\pi x \cdot \xi} \left[ e^{i2\pi y \cdot \xi} -e^{i2\pi z \cdot \xi}  \right]\tilde{p}(z, \xi)\varphi(\langle\xi\rangle/t) \diff \xi
        \end{equation*}
        \begin{equation*}
            =f(x-y, y, z, t) + g(x, y, z, t).
        \end{equation*}
        Then, we obtain 
        \begin{align*}
            \int_{\mathbb{T}^n}|g(x, y, z, t)|\diff x \leq & Ct^{-\rho n/2} \left[ \int_{\mathbb{T}^n} |g(x, y, z, t)|^2(1 + t^{2\rho}|x|^2)^N \diff x  \right]^{1/2}\\
            \leq & Ct^{-\rho n/2} \sum_{|\alpha|\leq N}  \left[ \int_{\mathbb{T}^n} |(t^\rho x)^\alpha g(x, y, z, t)|^2\diff x  \right]^{1/2}.
        \end{align*}
        Let us notice that $g$ is the Fourier transform on $\mathbb{R}^n$ on the first variable of the function $G(\xi, y, z, t) = \left[ e^{i2\pi y \cdot \xi} -e^{i2\pi z \cdot \xi}  \right]\tilde{p}(z, \xi)\varphi(\langle\xi\rangle/t)$. Also, as above, $|\partial^\gamma_\xi( e^{i2\pi y \cdot \xi} -e^{i2\pi z \cdot \xi}  )| \leq Ct\sigma$ when $t\sigma<1$. Hence, assuming $t\sigma < 1$ and using Plancherel's identity, we get
        \begin{align*}
            \int_{\mathbb{T}^n}|g(x, y, z, t)|\diff x \leq &Ct^{-\rho n/2} \sum_{|\alpha|\leq N} \left\| \partial^\alpha_\xi G(\xi, y, z ,t) \right\|_{L^2(\mathbb{R}^n)}\\ 
            \leq &Ct^{-\rho n/2} \sum_{\beta\leq\alpha,\;|\alpha|\leq N} \left\| \partial^\beta\xi \left[ (e^{i2\pi y \cdot \xi} -e^{i2\pi z \cdot \xi})  \tilde{p}(z, \xi)\right]  \partial^{\alpha-\beta}_\xi\varphi(\langle\xi\rangle/t)  \right\|_{L^2(\mathbb{R}^n)} \\
            \leq & Ct^{-\rho n/2} \sum_{\beta\leq\alpha,\;|\alpha|\leq N} t\sigma t^{m-\rho|\beta|}t^{-|\alpha-\beta|}t^{n/2} \\
            \leq& Ct\sigma t^{n(1-\rho)/2 + m} \leq Ct\sigma.
        \end{align*}
        Now, we drop the restriction $t\sigma<1$ and we notice that
        \begin{equation*}
            g(x, y, z, t) = \int_{\mathbb{R}^n} e^{-i2\pi (x-z) \cdot \xi} \left[ e^{i2\pi (y-z) \cdot \xi} -1 \right]\tilde{p}(z, \xi)\varphi(\langle\xi\rangle/t) \diff \xi.
        \end{equation*}
        Thus, we have that
        \begin{equation*}
            \int_{\mathbb{T}^n} |g(x, y, z, t)|\diff x \leq Ct^{-\rho N}\sigma^{\rho(n/2 - N)} \left[ \int_{\mathbb{T}^n} (t^{2\rho}|x-z|^2)^N|g(x, y, z, t)|^2\diff x \right].
        \end{equation*}
        We can use Plancherel's identity as above, to get
        \begin{align*}
            \int_{\mathbb{T}^n} |g(x, y, z, t)|\diff x \leq & Ct^{-\rho N}\sigma^{\rho(n/2 - N)} \sum_{\beta\leq\alpha,\;|\alpha|=N} \left\| \partial^\beta_\xi \left[ (e^{i2\pi (y-z) \cdot \xi} -1)\tilde{p}(z, \xi) \right] \partial^{\alpha-\beta}_\xi\varphi(\langle\xi\rangle/t) \right\|_{L^2(\mathbb{R}^n)} \\
            \leq& Ct^{-\rho N}\sigma^{\rho(n/2 - N)} \sum_{\beta\leq\alpha,\;|\alpha|=N}t^{-|\alpha-\beta|}t^{n/2} \\
            \leq & C(t\sigma)^{\rho(n/2-N)}.
        \end{align*}
    \end{itemize}
    We can use the same procedures to find these bounds for $f(x-y, y,z, t)$ by noticing that $\int_{|x-z|>2\sigma^\rho} |f(x-y, y,z, t)|\diff x \leq \int_{|x|>\sigma^\rho} |f(x, y,z, t)|\diff x$. Then the desired bound comes from the calculation in \cref{eq:lastbound-b}.
\end{proof}

\subsection{Boundedness results}
\label{section:boundedness}
Here we state and prove our main results regarding the boundedness of a certain classes of toroidal pseudo-differential operators.
\begin{remark}
    Upon inspecting the proof of the Euclidean case (\cref{theo:22boundR}) in \cite{hounie}, we can argue that the proof of the analogous case in $L^2(\mathbb{T}^n)$ (\cref{theo:22boundT}) is a faithful reproduction and does not need to be written in this document. However, we encourage the reader to check \cite{hounie} for more details.
\end{remark}

\begin{theorem}
    Let $\tilde{p}:\mathbb{T}^n \times \mathbb{R}^n \rightarrow \mathbb{C}$ be a symbol such that for $0<\rho\leq1$, $0\leq \delta < 1$, $m\leq -n \lambda = -n\max\{0, (\delta - \rho)/2\}$ and $|\alpha|,|\beta| \leq \lceil n/2\rceil$ satisfies: 
    \begin{equation}
        \left|\partial^\alpha_\xi \partial^\beta_x \tilde{p}(x, \xi)\right| \leq C_{\alpha\beta}\langle\xi\rangle^{m-\rho|\alpha| + \delta|\beta|}.
    \end{equation}

    Then $\mathrm{Op}(\tilde{p})$ is bounded from $L^2(\mathbb{T}^n)$ into $L^2(\mathbb{T}^n)$ with norm proportional to the best bound $C_{\alpha\beta}$.
    \label{theo:22boundT}
\end{theorem}
The previous result helps us prove an auxiliar result that will be useful in the proof of the main result of this paper.
\begin{theorem}
    Let $T \in \Psi^m_{\rho, \delta}(\mathbb{T}^n \times \mathbb{Z}^n) $, $0 < \rho \leq 1$, $0 \leq \delta < 1$, $m \leq - n [(1-\rho)/2 + \lambda] $, then $T$ is a continuous mapping
    \begin{itemize}
        \item[(a)]  from $L^2(\mathbb{T}^n)$ into $L^{2/\rho}(\mathbb{T}^n)$,\\
        \item[(b)] from $L^{2/(2-\rho)}(\mathbb{T}^n)$ into $L^{2}(\mathbb{T}^n)$.
    \end{itemize}
    \label{theo:Lbounds}
\end{theorem}

\begin{proof}
    First we notice that $J^{n(1-\rho)/2}T$ and $TJ^{n(1-\rho)/2}$ have order $\leq -n\lambda$, so they are bounded on $L^2(\mathbb{T}^n)$ by \cref{theo:22boundT}. Also, by the Hardy-Littlewood-Sobolev inequality, we have that $J^{-n(1-\rho)/2}$ is a continuous mapping from $L^2(\mathbb{T}^n)$ into $L^{2/\rho}(\mathbb{T}^n)$ and from $L^{2/(2-\rho)}(\mathbb{T}^n)$ into $L^2(\mathbb{T}^n)$. Hence 
    \begin{equation*}
        \| Tf \|_{2/\rho} = \| TJ^{n(1-\rho)/2}J^{-n(1-\rho)/2}f\|_{2/\rho} \leq C\| TJ^{n(1-\rho)/2}f \|_2 \leq C\|f\|_2 ,
    \end{equation*}
    and 
    \begin{equation*}
        \| Tf \|_{2} = \| J^{-n(1-\rho)/2}J^{n(1-\rho)/2}Tf\|_{2} \leq C\| J^{-n(1-\rho)/2}f \|_2 \leq C\| f \|_{2/(2-\rho)}.
    \end{equation*}
    Hence, proving the desired result.
\end{proof}
We now prove a result that will imply the boundedness of a certain class of toroidal pseudo-differential operators from $L^1(\mathbb{T}^n)$ into weak-$L^1(\mathbb{T}^n)$.
\begin{remark}
    Notice that the following theorem applies not only for pseudo-differential operators but operators with operator valued kernels, as in \cref{def:operator-kernel}, between Bochner-Lebesgue spaces $L^p(\mathbb{T}^n;X)$, as in \cref{def:bochner-lebesgue}, where $X$ is a Banach space. Also, in \cite[Theorem~4.1]{alvarez-milman}, the original statement in $\mathbb{R}^n$ requires the $D_{1,\alpha}$ condition, namely, for some sequence $(d_j) \in \ell^1(\mathbb{R}^n)$ we have:
    
    \begin{equation*}
        \int_{C_j(z, \sigma^\alpha)} \|k(x, y) - k(x, z)\|_{\mathcal{B}(X, Y)}\diff x \leq d_j|C_j(z, \sigma^\alpha)|,
    \end{equation*}
    where $C_j(z, \sigma^\alpha) = \{ x : 2^j\sigma^\alpha < |x - z| \leq 2^{j+1}\sigma^\alpha\}$ and $|y-z|<\sigma$.\\
    
    However the authors argue that this condition can be substituted by requiring the operators satisfy the estimates stated below.
\end{remark}

\begin{theorem}
    Let $T$ be be an operator with operator valued kernel, as in \cref{def:operator-kernel},  that extends to a bounded operator from $L^2(\mathbb{T}^n; X)$ into $L^2(\mathbb{T}^n; Y)$ and from $L^q(\mathbb{T}^n; X)$ into $L^2(\mathbb{T}^n; Y)$ such that for some $\alpha$ and $\beta$:
    \begin{equation}
        \frac{1}{q} = \frac{1}{2} + \frac{\beta}{n} \; , \; (1-\alpha)\frac{n}{2} \leq \beta < \frac{n}{2}.
        \label{eq:alpha_condition}
    \end{equation}
    Also assume that its kernel $k(x, y)$ satisfies the following condition when $|y-z|<\sigma$
    \begin{equation}
        \int_{|x-z|>c\sigma^\alpha}\|k(x, y)-k(x,z)\|_{\mathcal{B}(X, Y)}\diff x\leq C \;, \quad 0<\sigma < 1,
        \label{eq:hyp-sigma<1}
    \end{equation}
    \begin{equation}
        \int_{|x-z|>c\sigma}\|k(x, y)-k(x,z)\|_{\mathcal{B}(X, Y)}\diff x \leq C \;, \quad 1 \leq \sigma.
        \label{eq:hyp-sgima>1}
    \end{equation}
    Then the operator $T$ extends to a bounded operator from $L^1(\mathbb{T}^n; X)$ into weak-$L^{1}(\mathbb{T}^n; X)$
    \label{theo:L1-weak}
\end{theorem}

\begin{proof}
    Let $f \in L^1(\mathbb{T}^n; X)$, let $\lambda > 0$ and let us consider the Calderón-Zygmund decomposition at level $\lambda$. Thus 
    \begin{equation*}
        \Omega = \{x \in \mathbb{T}^n : Mf(x) > \lambda\} = \bigcup_{j=1}^\infty Q_j,
    \end{equation*}
    where $M$ is the maximal operator defined as

    \begin{equation*}
        Mf(x) = \sup_{B\ni x}\frac{1}{|B|}\int_B\|f(y)\|_X\diff y,
    \end{equation*} 
    where $B$ is a ball in $\mathbb{T}^n$. So that we can define $f = g + b$, here $f_{Q_j}$ is the mean value of $f$ over $Q_j$ and 
    \begin{equation*}
        g = f\chi_{\mathbb{T}^n \setminus \Omega} + \sum_{j=1}^\infty f_{Q_j}\chi_{Q_j},
    \end{equation*}
    \begin{equation*}
        b = \sum_{j=1}^\infty(f - f_{Q_j})\chi_{Q_j} = \sum_{j=1}^\infty b_j.
    \end{equation*}
    Moreover, 
    \begin{equation}
        \|g(x)\|_X \leq C\lambda \quad \text{and} \quad \|g\|_{L^1(\mathbb{T}^n; X)} \leq \|f\|_{L^1(\mathbb{T}^n; X)},
        \label{eq:1norm-X}
    \end{equation}
    
    \begin{equation}
        \int_{Q_j} \|b_j(x)\|_X \diff x \leq C|Q_j|\lambda, \quad  \quad \quad \int_{\mathbb{T}^n} b_j \diff x = 0,
        \label{eq:cz-condition}
    \end{equation}
    \begin{equation*}
        |\Omega| \leq \frac{C}{\lambda}\|f\|_{L^1(\mathbb{T}^n; X)}.
    \end{equation*}
    From \cref{eq:1norm-X} we have that $g \in L^2(\mathbb{T}^n; X)$, and also the inequality $\|g\|_{L^2(\mathbb{T}^n; X)}^2 \leq C\lambda \|f\|_{L^1(\mathbb{T}^n; X)}$. So, using Chebyshev's inequality and the $L^2$-boundedness of $T$ we get
    \begin{equation*}
        \lambda^2 |\{x \in \mathbb{T}^n : \|Tg(x)\|_Y > \lambda/2\}| \leq C\|Tg\|_{L^2(\mathbb{T}^n; Y)}^2 \leq C\|g\|_{L^2(\mathbb{T}^n; X)}^2 \leq C\lambda\|f\|_{L^1(\mathbb{T}^n; X)}.
    \end{equation*}
    On the other hand, if we set $c\Omega$ as the union of the cubes $cQ_j$ with the same center and scaled side length, we get
    \begin{equation*}
        |\{ x \in c\Omega : \|Tb(x)\|_Y > \lambda/2  \}| \leq |c\Omega| \leq \frac{C}{\lambda}\|f\|_{L^1(\mathbb{T}^n; X)}.
    \end{equation*}
    Let $\sigma_j$ be the diameter of $Q_j$ and
    \begin{equation*}
        F = \sum_{\sigma_j < 1} b_j\;, \quad  \quad G = \sum_{\sigma_j \geq 1}b_j,
    \end{equation*}
    then for suitable $c>0$ and $Q_j$, centered at $z_j$, we get $|x - z_j| > 4\sigma_j$, when $x \in \mathbb{T}^n \setminus c\Omega$. \\
    Using Chebyshev's inequality and \cref{eq:cz-condition} we get by \cref{eq:hyp-sgima>1} the following estimates
    \begin{align*}
       & \lambda|\{ x \in \mathbb{T}^n \setminus c\Omega : \|TG(x)\|_Y > \lambda/4 \}| \\
       \leq & C \int_{\mathbb{T}^n \setminus c\Omega} \|TG(x)\|_Y \diff x \\
        \leq & 
        C\sum_{\sigma_j \geq 1} \int_{|x-z_j|>4\sigma_j} \int_{Q_j} \|k(x, y) - k(x, z_j)\|_{\mathcal{B}(X, Y)}\|b_j(y)\|_X \diff y \diff x \\
        \leq & C\sum_{\sigma_j \geq 1} \int_{Q_j}\|b_j(y)\|_X \diff y \leq C\|f\|_{L^1(\mathbb{T}^n; X)}.
    \end{align*}
    Now, let $\varphi$ be a test funciton supported on $\{x \in \mathbb{T}^n : |x| \leq 1/c\}$ and such that $\int\varphi(x) \diff x = 1$ with $\varphi \geq 0$. Let us define 
    \begin{equation*}
        \varphi_j(x) = \frac{1}{\sigma_j^{n/\alpha}} \varphi\left( \frac{x}{\sigma_j^{1/\alpha}} \right).
    \end{equation*}
    We write 
    \begin{equation*}
        F = \sum_{\sigma_j < 1}b_j *\varphi_j + \sum_{\sigma_j<1}(b_j - b_j * \varphi_j) = F' + F'' .
    \end{equation*}
    Then, for $x \in\mathbb{T}^n \setminus c\Omega $, we have that
    \begin{align*}
        T(b_j - b_j * \varphi_j)(x) = & \int_{\mathbb{T}^n}k(x, y)b_j(y)\diff y - \int_{\mathbb{T}^n}k(x, w) \int_{\mathbb{T}^n} \varphi_j(w-y)b_j(y)\diff y \diff w \\
        = & \int_{\mathbb{T}^n}\left[\int_{\mathbb{T}^n}[k(x, y) - k(x, w)]\varphi_j(w-y)\diff w\right] b_j(y)\diff y.
    \end{align*}
    Hence, using Chebyshev's inequality we obtain
    \begin{equation}
        \lambda |\{x \in \mathbb{T}^n \setminus c\Omega : \|TF''(x)\|_Y > \lambda/8\}| 
        \label{eq:TF''}
    \end{equation}
    \begin{equation*}
        \leq  C\int_{\mathbb{T}^n \setminus c\Omega}\|TF''(x)\|_Y \diff x
    \end{equation*}
    \begin{equation*}
        \leq  C \sum_{\sigma_j< 1}\int_{\mathbb{T}^n}\left[\int_{\mathbb{T}^n}\|k(x, y) - k(x, w)\|_{\mathcal{B}(X,Y)}\varphi_j(w-y)\diff w\right] \|b_j(y)\|_X\diff y.
    \end{equation*}
    When $x \in \mathbb{T}^n \setminus c\Omega$ and $|y - w| < \sigma_j^{1/\alpha}$ we have that
    \begin{equation*}
        |x-w| \geq |x-z_j| - |z_j - y| - |y - w| > 4\sigma_j - \sigma_j - \sigma_j^{1/\alpha} > 2\sigma_j.
    \end{equation*}

    So we can estimate \cref{eq:TF''} by 
    \begin{equation*}
          \sum_{\sigma_j< 1}\int_{Q_j}\left[\int_{|w-y|<\sigma_j^{1/\alpha}}\int_{|x-w|>2\sigma_j}\|k(x, y) - k(x, w)\|_{\mathcal{B}(X,Y)}\diff x \diff w\right] \|b_j(y)\|_X\diff y
    \end{equation*}
    \begin{equation*}
        \leq C\sum_{\sigma_j<1} \int_{Q_j}\|b_j(y)\|_X\diff y  \leq C\|f\|_{L^1(\mathbb{T}^n; X)}.
    \end{equation*}
    Now, we only have to prove the inequality
    \begin{equation*}
    \lambda |\{ x \in \mathbb{T}^n \setminus c\Omega : \|TF'(x)\|_Y > \lambda/8 \}| \leq C\|f\|_{L^1(\mathbb{T}^n; X)},
    \end{equation*}
    which would be proven if we obtain the estimate 
    \begin{equation}
        \|J^{-\beta} F'\|_{L^2(\mathbb{T}^n; X)}^2 \leq A\lambda \|f\|_{L^1(\mathbb{T}^n; X)},
    \end{equation}
    where $J$ is the Bessel potential of order one. Since $\beta$ satisfies \cref{eq:alpha_condition} we have  that $TJ^{\beta}$ is bounded on $L^2(\mathbb{T}^n;Y)$ and
    \begin{equation*}
        \|T F'\|_{L^2(\mathbb{T}^n; Y)}^2 =\|TJ^{\beta}J^{-\beta }F'\|_{L^2(\mathbb{T}^n; Y)}^2 \leq  C\|J^{-\beta} F'\|_{L^2(\mathbb{T}^n; X)}^2 \leq CA\lambda \|f\|_{L^1(\mathbb{T}^n; X)}.
    \end{equation*}
    Hence, we can use  Chebyshev's inequality to obtain
    \begin{equation*}
        \lambda^2 |\{ x \in \mathbb{T}^n \setminus c\Omega : \|TF'(x)\|_Y > \lambda/8 \}| \leq C \int_{\mathbb{T}^n \setminus c\Omega} \|TF'(x)\|_Y^2\diff x  \leq CA\lambda\|f\|_{L^1(\mathbb{T}^n; X)}.
    \end{equation*}
    We recall that
    \begin{equation*}
        F' = \sum_{\sigma_j < 1}f\chi_{Q_j} * \varphi_j - \sum_{\sigma_j < 1}f_{Q_j}\chi_{Q_j}*\varphi_j.
    \end{equation*}
    We define $x \sim Q_j$ if $x$ belongs to the closure of any cube $Q_k$ adjacent to $Q_j$. Thus, 
    \begin{equation*}
        J^{-\beta}\sum_{\sigma_j < 1}f\chi_{Q_j} * \varphi_j (x) = \sum_{x\sim Q_j}J^{-\beta} (f\chi_{Q_j} * \varphi_j )(x) + \sum_{x\nsim Q_j}J^{-\beta} (f\chi_{Q_j} * \varphi_j )(x)  = F_1(x) + F_2(x).
    \end{equation*}
    So that 
    \begin{equation*}
        (J^{-\beta} F')(x) = F_1(x) + F_2(x) - \sum_{\sigma_j < 1}f_{Q_j}J^{-\beta} (\chi_{Q_j}*\varphi_j)(x).
    \end{equation*}
    Using Fefferman's proof adapted to compact Lie groups \cite{cardona-ruzhansky}, we have that for $x \nsim Q_j$ 
    \begin{equation*}
        \left\|J^{-\beta} (f\chi_{Q_j} *\varphi_j)(x)\right\|_X \leq CJ^{-\beta} \|f_{Q_j}\|_X (\chi_{Q_j}*\varphi_j)(x).
    \end{equation*}
    By \cref{eq:cz-condition} and the fact that $J^{-\beta}$ maps positive functions into positive functions we have that
    \begin{align*}
        \left\| F_2(x) - \sum_{\sigma_j<1} f_{Q_j}J^{-\beta}(\chi_{Q_j} * \varphi_j)(x) \right\|_X = & \left\| \sum_{\sigma_j<1} J^{-\beta}[(f\chi_{Q_j} - f_{Q_j}\chi_{Q_j}) * \varphi_j](x) \right\|_X \\
        \leq & \left\| \sum_{\sigma_j<1} J^{-\beta}(b_j * \varphi_j)(x) \right\|_X \\ 
        \leq & C\lambda \sum_{\sigma_j<1} J^{-\beta} (\chi_{Q_j}*\varphi_j)(x)\\
        \leq & C\lambda \|J^{-\beta}\|_1\left\| \sum_{\sigma_j <1}\chi_{Q_j} *\varphi_j \right\|_{L^\infty(\mathbb{T}^n; X)} \leq C\lambda ,
    \end{align*}
    since the supports of $ \chi_{Q_j} * \varphi_j $ have finite overlapping. On the other hand,
    \begin{align*}
        \left\| F_2 - \sum_{\sigma_j < 1} J^{-\beta} f_{Q_j}\chi_{Q_j} * \varphi_j \right\|_{L^1(\mathbb{T}^n; X)} \leq & \sum_{\sigma_j<1} \left\|J^{-\beta}[\varphi_j * (f - f_{Q_j})\chi_{Q_j}]\right\|_{L^1(\mathbb{T}^n; X)} \\
        \leq & C\sum_{\sigma_j < 1}\int_{Q_j}\|f(y)\|_X \diff y \leq C\|f\|_{L^1(\mathbb{T}^n; X)}.
    \end{align*}
    Combining these two estimates we obtain 
    \begin{equation*}
        \left\| F_2 - \sum_{\sigma_j < 1} J^{-\beta} f_{Q_j}\chi_{Q_j} * \varphi_j \right\|_{L^2(\mathbb{T}^n; X)}^2 \leq C\lambda\|f\|_{L^1(\mathbb{T}^n; X)} .
    \end{equation*}
    Now it only remains to prove the inequality $\|F_1\|_{L^2(\mathbb{T}^n; X)}^2 \leq C\lambda \|f\|_{L^1(\mathbb{T}^n; X)}$, which can be done as in Fefferman's proof, \cite{cardona-ruzhansky}. For fixed $x \in \mathbb{T}^n$, let 
    \begin{equation*}
        F_1^j(x) = \begin{cases}
            J^{-\beta} f\chi_{Q_j}*\varphi_j(x) & \text{if }x \nsim Q_j \\ 0 & \text{if not}.
        \end{cases}
    \end{equation*}
    Thus, $F_1(x)=\sum_{\sigma_j <1} F_1^j(x)$ and $F_1^j(x) \neq 0$ for at most $N$ values. Then,
    \begin{align*}
        \|F_1(x)\|_Y^2 \leq & \left( \sum_{j=1}^N\left\|F_1^{j(x)}(x)\right\|_Y \right)^2 \\
        \leq & \sum_{j,h=1}^N\left\|F_1^{j(x)}(x)\right\|_Y \left\|F_1^{h(x)}(x)\right\|_Y \\
        \leq & 2\left( \sum_{j,h=1}^N \left\|F_1^{j(x)}(x)\right\|_Y^2 + \left\|F_1^{h(x)}(x)\right\|_Y^2 \right) \\
        \leq & 4N\sum_{\sigma_j < 1}\left\|F_1^j(x)\right\|_Y^2.
    \end{align*}
    Hence, using Hausdorff-Young's inequality we have that
    \begin{equation*}
        \|F_1\|_{L^2(\mathbb{T}^n; X)}^2 \leq 4N\sum_{\sigma_j<1} \left\| F_1^j \right\|_{L^2(\mathbb{T}^n; X)}^2 \leq 4N\sum_{\sigma_j<1} \left\| J^{-\beta} \varphi_j \right\|_{L^2(\mathbb{T}^n; X)}^2\left\| f\chi_{Q_j} \right\|_{L^1(\mathbb{T}^n; X)}^2.
    \end{equation*}
    On the other hand, by Plancherel's identity and since $-2\beta < -n$ we get
    \begin{equation*}
        \left\| J^{-\beta} \varphi_j \right\|_{L^2(\mathbb{T}^n; X)}^2 = \sum_{\xi \in \mathbb{Z}^n } \langle\xi\rangle^{-2\beta} \left|(\mathcal{F}_{\mathbb{T}^n} \varphi)\left(\sigma_j^{1/\alpha}\xi\right)\right|^2 \leq     \frac{C}{|Q_j|},
    \end{equation*}
    since $|Q_j| \leq 1$. Finally
    \begin{equation*}
        \|F_1\|_{L^2(\mathbb{T}^n; X)}^2 \leq C\sum_{\sigma_j<1}\frac{1}{|Q_j|} \|f\chi_{Q_j}\|_{L^1(\mathbb{T}^n; X)}^2 \leq C\lambda \|f\|_{L^1(\mathbb{T}^n; X)}.
    \end{equation*}
    Thus, completing the proof.
\end{proof}

Now, we state the main result of this paper.

\begin{theorem}
    Let $T \in \Psi^m_{\rho, \delta}(\mathbb{T}^n \times \mathbb{Z}^n) $, $0 < \rho \leq 1$, $0 \leq \delta < 1$, $m \leq - n [(1-\rho)/2 + \lambda] $, then $T$ and its adjoint $T^*$ are continuous mappings
    \begin{itemize}
        \item[(a)] from the Hardy space $H^1(\mathbb{T}^n)$ into $L^1(\mathbb{T}^n)$,\\
        \item[(b)] from $L^1(\mathbb{T}^n)$ into weak-$L^1(\mathbb{T}^n)$,\\
        \item[(c)] from $L^\infty(\mathbb{T}^n)$ into $\text{BMO}(\mathbb{T}^n)$.
    \end{itemize}
\end{theorem}

\begin{proof}
    \begin{itemize}
        \item[(a)] Let $a$ be an $H^1(\mathbb{T}^n)$-atom supported in $B(z, \sigma)$, satisfying $\|a\|_\infty\leq |B|^{-1}$ and the cancellation condition. If $\sigma < 1$, set $B'=B(z, 2\sigma^\rho)$ and $A = \mathbb{T}^n\setminus B'$. Then 
        \begin{equation*}
            \int_{\mathbb{T}^n}|Ta(x)|\diff x \leq \int_{B'}|Ta(x)|\diff x  + \int_{A}|Ta(x)|\diff x = I_1 + I_2.
        \end{equation*}
        Now, using \cref{theo:Lbounds} (b) we get 
        \begin{align*}
            I_1 \leq & \|\chi_{B'}\|_{2}\|Ta\|_{{2}}\leq C\sigma^{\rho n/2}\|a\|_{{2/(2-\rho)}} \\ \leq & C\sigma^{\rho n/2} \left[ \int_B |B|^{-2/(2-\rho)}\diff x\right]^{(2-\rho)/2} 
            \leq  C\sigma^{\rho n/2}|B|^{-\rho/2} \leq C,
        \end{align*}
        and using \cref{theo:sigma-kernel-estimate} (b) we have the estimate,
        \begin{equation*}
            I_2 \leq \int_A\int_B |k(x, y) - k(x, z)||a(y)|\diff y \diff x \leq \sup_{|y-z|<\sigma}\int_A|k(x,y)-k(x, z)|\diff x \leq C.
        \end{equation*}
        When $\sigma \geq 1$, we set $B'=B(z, 2\sigma)$ and $A = \mathbb{T}^n\setminus B'$. Then we split the $L^1(\mathbb{T}^n)$-norm into $I_1 + I_2$ as above. Now we use \cref{theo:Lbounds} (a) to get
        \begin{align*}
            I_1 \leq & \|\chi_{B'}\|_{2}\|Ta\|_{{2}}\leq C\sigma^{ n/2}\|a\|_{{2}} \\ 
            \leq & C\sigma^{ n/2} \left[ \int_B |B|^{-2}\diff x\right]^{1/2}  \leq C\sigma^{n/2}|B|^{-1/2}\leq C,
        \end{align*}
        and we use \cref{theo:sigma-kernel-estimate} (a) to estimate $I_2$.\\

        \item[(b)] Notice that by \cref{theo:22boundT} we can set $q$ as $2/(2-\rho)$ in the hypothesis of \cref{theo:L1-weak}. Also, setting $\alpha$ as $\rho$ and $\beta$ as $n(1-\rho)/2$ satisfy the condition in  \cref{eq:alpha_condition}. Then, by \cref{theo:sigma-kernel-estimate} we get the required conditions  to prove the result.\\

        \item[(c)] Let $B = B(z, \sigma)$. If $\sigma < 1$, set $B'=B(z, 2\sigma^\rho)$ and $A = \mathbb{T}^n\setminus B'$ so we get $f = f\chi_{B'} + f\chi_A = f_1 + f_2$. Now, set $b = Tf_2(z)$, which is well defined since $Tf_2$ is smooth in $B'$. Then 
        \begin{equation*}
            \frac{1}{|B|}\int_B|Tf(x) - b| \diff x \leq \frac{1}{|B|}\int_B|Tf_1(x)|\diff x + \frac{1}{|B|}\int_B|Tf_2(x) -b|\diff x = I_1 + I_2.
        \end{equation*}
        By \cref{theo:Lbounds} (a) we have the inequalities
        \begin{equation*}
            I_1 \leq \left\|\frac{1}{|B|}\chi_{B'}\right\|_{2/(2-\rho)}\|Tf_1\|_{2/\rho}\leq C|B|^{-\rho/2}\|f_1\|_2 \leq C|B|^{-\rho/2}\left[\int_{B'}\|f\|^2_\infty\diff x\right]^{1/2} \leq C \|f\|_\infty,
        \end{equation*}
        and using \cref{theo:sigma-kernel-estimate} (c) we get 
        \begin{equation*}
            I_2 \leq \frac{1}{|B|}\int_B\int_A|k(x, y) - k(z, y)||f(y)|\diff y \diff x \leq \sup_{|y-z|<\sigma}\int_A|k(x, y) - k(z, y)|\|f\|_\infty \diff y \leq C \|f\|_\infty.
        \end{equation*}
        When $\sigma \geq 1$, we set $B'=B(z, 2\sigma)$ and $A = \mathbb{T}^n\setminus B'$. Then we split the BMO-norm into $I_1 + I_2$ as above. Now we use \cref{theo:Lbounds} (b) to get
        \begin{equation*}
            I_1 \leq \left\|\frac{1}{|B|}\chi_B\right\|_2\left\|Tf_1\right\|_2 = |B|^{-1/2}\|Tf_1\|_2 \leq C|B|^{-1/2}\left[\int_{B'}\|f\|^2_\infty\diff x\right]^{1/2} \leq C \|f\|_\infty,
        \end{equation*}
        and we use \cref{theo:sigma-kernel-estimate} (a) to estimate $I_2$. Thus, we get the inequality
        \begin{equation*}
            \|Tf\|_{\mathrm{BMO}(\mathbb{T}^n)} = \sup_B \inf_{b\in \mathbb{C}} \frac{1}{|B|} \int_B |Tf(x) - b|\diff x \leq C\|f\|_{L^\infty(\mathbb{T}^n)}.
        \end{equation*}
        Thus, completing the proof for $T$. 
    \end{itemize}  
    Now, notice that the only properties of $T$ that were used in this proof are the order of the operator and the fact the its kernel satisfies \cref{theo:sigma-kernel-estimate}, properties that also apply to $T^*$. Hence the result follows for $T^*$ as well.
\end{proof}
Now we can employ the Fefferman-Stein complex interpolation argument, see \cite[pp.~159]{fefferman-stein}, to obtain the boundedness of pseudo-differential operators from $L^p(\mathbb{T}^n)$ into itself.
\begin{theorem}
    Let $T \in \Psi^m_{\rho, \delta}(\mathbb{T}^n \times \mathbb{Z}^n) $, $0 < \rho \leq 1$, $0 \leq \delta < 1$ and  
    \begin{equation}
    m \leq - n \left[(1-\rho)\left|\frac{1}{p} - \frac{1}{2}\right| + \lambda\right] .
    \label{eq:Lp-restriction}
    \end{equation} 
    Then $T$  is a continuous mapping from $L^p(\mathbb{T}^n)$ into itself.
    \label{theo:Lp-bounds}
\end{theorem}

\begin{proof}
    We can use the complex interpolation argument between $(H^1(\mathbb{T}^n),\; L^1(\mathbb{T}^n))$ and  $(L^2(\mathbb{T}^n),\; L^2(\mathbb{T}^n))$ for $1 < p < 2$, and between $(L^2(\mathbb{T}^n),\; L^2(\mathbb{T}^n))$ and the pair $(L^1(\mathbb{T}^n),\; \mathrm{BMO}(\mathbb{T}^n))$ for $2 < p<\infty$. In fact, $T$ is bounded on $L^2(\mathbb{T}^n)$ if $m\leq -n\lambda$ by \cref{theo:22boundT}. On the other hand, if $m\leq -n[(1-\rho)/2 + \lambda]$, then $T$ will be bounded from $H^1(\mathbb{T}^n)$ into $L^1(\mathbb{T}^n)$ and from $L^\infty(\mathbb{T}^n)$ into $\mathrm{BMO}(\mathbb{T}^n)$. Then $T$ will be bounded on $L^p(\mathbb{T}^n)$, by the Fefferman-Stein interpolation argument, for
    \begin{equation*}
        \frac{1}{p} = \frac{1-\theta}{q} + \frac{\theta}{2},
    \end{equation*}
    and $0<\theta<1$, with $q=1$ or $q=\infty$. Which is equivalent to the restriction 
    \begin{equation*}
        m \leq -n\lambda\theta - n\left[(1-\rho)/2 +\lambda\right](1 - \theta).
    \end{equation*}
    Namely, that $m$ satisfies \cref{eq:Lp-restriction}, completing the proof.
\end{proof}
We can use the properties of Bessel potential operators to extend the $L^p(\mathbb{T}^n)$-boundedness result to the $L^p(\mathbb{T}^n)$-$L^q(\mathbb{T}^n)$ case. In the case of the torus, the following theorem extends the $L^p(\mathbb{T}^n)$-$L^q(\mathbb{T}^n)$-boundedness result in  \cite{cardona-delgado-kumar} to the complete range $0 < \rho \leq 1$, $0 \leq \delta < 1$.
\begin{theorem}
    Let $T \in \Psi^m_{\rho, \delta}(\mathbb{T}^n \times \mathbb{Z}^n) $, $0 < \rho \leq 1$, $0 \leq \delta < 1$, then $T$  is a continuous mapping from $L^p(\mathbb{T}^n)$ into $L^q(\mathbb{T}^n)$ for $1<p\leq q<\infty$ when 
    \begin{itemize}
        \item[(a)] if $1<p\leq 2\leq q$ and 
        \begin{equation}
            m\leq -n \left( \frac{1}{p} - \frac{1}{q} + \lambda \right),
            \label{eq:Lp-Lq-bounds-a}
        \end{equation}
        \item[(b)] if $2\leq p \leq q$ and 
        \begin{equation}
            m \leq -n \left[ \frac{1}{p} - \frac{1}{q} + (1-\rho)\left(\frac{1}{2} - \frac{1}{p}\right) + \lambda \right],
            \label{eq:Lp-Lq-bounds-b}
        \end{equation}
     
        \item[(c)] if $p\leq q\leq 2$ and  
        \begin{equation}
            m \leq -n \left[ \frac{1}{p} - \frac{1}{q} + (1-\rho)\left(\frac{1}{q} - \frac{1}{2}\right) + \lambda \right].
            \label{eq:Lp-Lq-bounds-c}
        \end{equation}
    \end{itemize}
    \label{theo:Lp-Lq-bounds}
\end{theorem}

\begin{proof}
    \begin{itemize}
        \item[(a)] Let $m_1 = -n(1/p - 1/2)$ and $m_2=-n(1/2 - 1/q)$. Then $m\leq m_1+m_2-n\lambda$ and $J^{-m_2}TJ^{-m_1}$ is bounded on $L^2(\mathbb{T}^n)$ by \cref{theo:22boundT}. Also, by the Hardy-Littlewood-Sobolev inequality, we get that $J^{m_1}$ is bounded from $L^p(\mathbb{T}^n)$ into $L^2(\mathbb{T}^n)$ and $J^{m_2}$ is bounded from $L^q(\mathbb{T}^n)$ into $L^2(\mathbb{T}^n)$. Hence 
        \begin{equation*}
            \| J^{m_2}(J^{-m_2}TJ^{-m_1})J^{m_1}f \|_q \leq C\|(J^{-m_2}TJ^{-m_1})J^{m_1}f  \|_2 \leq C\| J^{m_1}f \|_2 \leq C \| f\|_p.
        \end{equation*}
        Thus, proving the result.\\

        \item[(b)] Set $m' = -n(1/p - 1/q)$, so that $J^{m'}$ is continuous from  $L^p(\mathbb{T}^n)$ into $L^q(\mathbb{T}^n)$ and $J^{-m'}T$ has order $m-m' \leq -n[(1-\rho)(1/2 - 1/q) + \lambda]$ and is a continuous mapping from $L^p(\mathbb{T}^n)$ into itself. Thus 
        \begin{equation*}
            \| J^{m'}(J^{-m'}T)f \|_q \leq C\|(J^{-m'}T)f\|_p \leq C\|f\|_p.
        \end{equation*}
        Hence, obtaining the desired estimate.\\

        \item[(c)] As above, we set  $m' = -n(1/p - 1/q)$ so that $TJ^{-m'}$ countinously maps  $L^q(\mathbb{T}^n)$ into itself and
        \begin{equation*}
            \|(TJ^{-m'})J^{m'}f\|_q \leq C\|J^{m'}f\|_q \leq C\|f\|_p.
        \end{equation*}
        Thus, completing the proof.
    \end{itemize}
\end{proof}
\begin{remark}
    Please notice that \cref{eq:Lp-Lq-bounds-b} reduces to \cref{eq:Lp-Lq-bounds-a} when $p=2$ and \cref{eq:Lp-Lq-bounds-c} reduces to \cref{eq:Lp-Lq-bounds-a} when $q=2$. Moreover, \cref{theo:Lp-Lq-bounds} reduces to \cref{theo:Lp-bounds} when $p=q$.
\end{remark}

\bibliographystyle{amsplain}

\end{document}